\documentclass[12pt]{amsart}

\usepackage{amsfonts,amsmath,amssymb,amsxtra,amsthm,tikz}
\usepackage{mathrsfs}
\usepackage{times}
\usepackage{anysize}
\usepackage{graphicx}
\marginsize{1in}{1in}{1in}{1in}

\newtheorem{theorem}{Theorem}[section]
\newtheorem{lemma}[theorem]{Lemma}

\newtheorem{corollary}[theorem]{Corollary}
\newtheorem{proposition}[theorem]{Proposition}

\theoremstyle{definition}
\newtheorem{definition}[theorem]{Definition}
\newtheorem{example}[theorem]{Example}

\theoremstyle{remark}
\newtheorem{remark}[theorem]{Remark}

\DeclareMathOperator{\In}{{\bf in}} 
\DeclareMathOperator{\Ter}{{\bf ter}}
\DeclareMathOperator{\supp}{supp}
\DeclareMathOperator{\Gap}{Gap} 
\DeclareMathOperator{\Span}{Span} 
\DeclareMathOperator{\Hom}{Hom} 
\DeclareMathOperator{\Ext}{Ext} 
\DeclareMathOperator{\Id}{Id}
\DeclareMathOperator{\PF}{PF}
\newcommand{\Rep}{\mathbf{Rep}}
\newcommand{\CB}{\mathcal{B}}

\title{A braid group action on parking functions} 
\author{Eugene Gorsky}
\address{Department of Mathematics, Columbia University.\newline
2990 Broadway New York, NY 10027.}
\email{egorsky@math.columbia.edu}

\author{Mikhail Gorsky}
\address{Steklov Mathematical Institute\newline 
8 Gubkina Street, Moscow, Russia 119991.}
\address{Universit\'e Paris Diderot -- Paris 7\newline
75205 PARIS CEDEX 13}
\email{mgorsky@math.jussieu.fr} 
 

\begin{document}

\begin{abstract}
We construct an action of the braid group on $n$ strands on the set of parking functions of $n$ cars
such that elementary braids have orbits of length 2 or 3. The construction is motivated by a theorem of  Lyashko and  Looijenga stating that the number of the distinguished bases for $A_n$ singularity equals $(n+1)^{n-1}$ and thus equals the number of parking functions. We construct an explicit bijection between the set of parking functions and the set of distinguished bases, which allows us to translate the braid group action on distinguished bases in terms of
parking functions. 
\end{abstract}

\maketitle

\section{Introduction}

The distinguished bases in the vanishing cohomology of the complex hypersurface singularity were introduced by  Gabrielov and  Lazzeri \cite{gabrielov,lazzeri} who followed the ideas of  Milnor \cite{milnor}. They were widely studied in singularity theory (e.g.  \cite{gusein,gusein2}), for more complete references see \cite{book} and citations therein.
 
Geometric origin of distinguished bases made them useful in mathematical physics (e.g. \cite{vc}). In \cite{lya} and \cite{loo}, O. Lyashko and E. Looijenga independently computed the number of distinguished bases for simple singularities as the degree of a certain covering.
For example, the singularity of type $A_n$ has $(n+1)^{n-1}$ distinguished bases, where basic vectors are considered up to a sign.

In this note, we use combinatorial structures to describe the distinguished bases for $A_n$ singularities. It is well known (e.g \cite{book}) that the vanishing cohomology for these singularities corresponds to the $A_n$ root lattice, which carries a non-symmetric bilinear Seifert form $\left\langle  \cdot,\cdot\right\rangle$ such that
$$\left\langle x, y \right\rangle + \left\langle y,x \right\rangle=(x,y).$$
The following theorem gives a combinatorial description of distinguished bases.

\begin{theorem}(\cite{gusein})
A basis in $A_n$ root lattice is distinguished if and only if the basis vectors are roots of this lattice, and the matrix of the Seifert form is upper-triangular in this basis.
\end{theorem}

\begin{corollary}
Switching a sign of some vectors in a distinguished basis will transform it to another distinguished basis.
\end{corollary}

\begin{definition} 
A {\it parking}, or  {\it preference} function on $n$ elements is a function
$$
f : \left\{1, \ldots, n\right\} \rightarrow \left\{1, \ldots, n\right\}
\ \mbox{\rm such that}\ |f^{-1}(\left\{1, \ldots, k\right\})| \geq k\ \forall k.
$$
We denote the set of all parking functions by $\PF_n$.
\end{definition}

Parking functions were introduced by Konheim and  Weiss in \cite{kw}, and were studied in different combinatorial and algebraic setups (e.g. \cite{dots,haglund,stanley}).
It is well known that the number of the parking functions of order $n$ equals to  $(n+1)^{n-1}$.

\begin{definition}
Let $e_i$ denote the positive simple roots of the $A_n$ root system,
and let $e=e_{i}+\ldots+e_{j}$ be a root.
We define its ``initial point'' as $\In(e)=i$.
For a basis $A=\{a_1,\ldots, a_n\},$ we define the ``initial vector'' 
as an integer sequence
$\In(A)=(\In(a_1),\ldots,\In(a_n)).$
\end{definition}

The function $\In(e)$ can be understood a filtration on the Cartan subalgebra of type $A_n$.  
The following theorem shows that one can reconstruct a distinguished basis from the
filtration levels of basic vectors.

\begin{theorem}
The ``initial vector'' map is a bijection between the set of distinguished bases (of positive roots) and the set of parking functions. 
\end{theorem}

In Theorem \ref{T2}, we also provide a simple geometric procedure of reconstruction of a distinguished basis from a parking function.
We illustrate it on several examples and apply this procedure for permutations and non-decreasing parking functions.

The correspondence between parking functions and distinguished bases seems to be quite surprising since the set of distinguished bases carries a natural
action of the braid group (\cite{gabrielov,book}) while the set of parking functions carries a natural action of the symmetric group. 
We discuss the induced action of the braid group on the parking functions in Section \ref{sec:braid} and prove the following result:

\begin{theorem}
There exists an action of braid group $\CB_n$ on the set $\PF_n$. All orbits of the action of elementary braids either have length 2 or length 3. 
\end{theorem}

We illustrate this action for $n=3$ in Figure \ref{fig:pf for a3}. The elementary braids $\alpha_1$ and $\alpha_2$ act along solid and dashed arrows respectively, and one can check that the braid relation $\alpha_1\alpha_2\alpha_1=\alpha_2\alpha_1\alpha_2$ is satisfied.

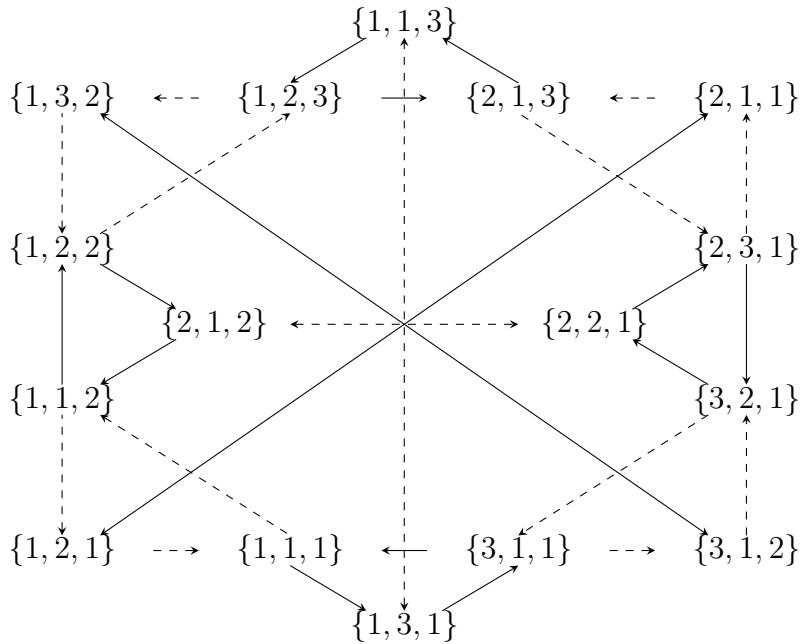
\begin{figure}[ht]
\begin{tikzpicture}
\draw (0,0) node {$\{1,2,1\}$};
\draw (3,0) node {$\{1,1,1\}$};
\draw (6,0) node {$\{3,1,1\}$};
\draw (9,0) node {$\{3,1,2\}$};
\draw (0,2) node {$\{1,1,2\}$};
\draw (0,4) node {$\{1,2,2\}$};
\draw (0,6) node {$\{1,3,2\}$};
\draw (3,6) node {$\{1,2,3\}$};
\draw (6,6) node {$\{2,1,3\}$};
\draw (9,6) node {$\{2,1,1\}$};
\draw (9,4) node {$\{2,3,1\}$};
\draw (9,2) node {$\{3,2,1\}$};
\draw (4.5,-1) node {$\{1,3,1\}$};
\draw (4.5,7) node {$\{1,1,3\}$};
\draw (2,3) node {$\{2,1,2\}$};
\draw (7,3) node {$\{2,2,1\}$};

\draw [->,>=stealth,dashed] (1.2,0)--(1.8,0);
\draw [<-,>=stealth] (4.2,0)--(4.8,0);
\draw [->,>=stealth,dashed] (7.2,0)--(7.8,0);
\draw [<-,>=stealth,dashed] (1.2,6)--(1.8,6);
\draw [->,>=stealth] (4.2,6)--(4.8,6);
\draw [<-,>=stealth,dashed] (7.2,6)--(7.8,6);
\draw [<-,>=stealth,dashed] (0,0.2)--(0,1.8);
\draw [->,>=stealth] (0,2.2)--(0,3.8);
\draw [<-,>=stealth,dashed] (0,4.2)--(0,5.8);
\draw [->,>=stealth,dashed] (9,0.2)--(9,1.8);
\draw [<-,>=stealth] (9,2.2)--(9,3.8);
\draw [->,>=stealth,dashed] (9,4.2)--(9,5.8);
\draw [<-,>=stealth,dashed] (0.5,1.8)--(3,0.2);
\draw [->,>=stealth,dashed] (0.5,4.2)--(3,5.8);
\draw [->,>=stealth,dashed] (8.5,1.8)--(6,0.2);
\draw [<-,>=stealth,dashed] (8.5,4.2)--(6,5.8);
\draw [<->,>=stealth] (0.5,0.2)--(8.5,5.8);
\draw [<->,>=stealth] (0.5,5.8)--(8.5,0.2);
\draw [<->,>=stealth,dashed] (4.5,-0.8)--(4.5,6.8);
\draw [<->,>=stealth,dashed] (3,3)--(6,3);
\draw [->,>=stealth] (3,-0.2)--(4,-0.8);
\draw [->,>=stealth] (5,-0.8)--(6,-0.2);
\draw [<-,>=stealth] (3,6.2)--(4,6.8);
\draw [<-,>=stealth] (5,6.8)--(6,6.2);
\draw [<-,>=stealth] (0.5,2.2)--(1.5,2.8);
\draw [<-,>=stealth] (1.5,3.2)--(0.5,3.8);
\draw [->,>=stealth] (8.5,2.2)--(7.5,2.8);
\draw [->,>=stealth] (7.5,3.2)--(8.5,3.8);
\end{tikzpicture}
\caption{Braid group action for $\PF_3$: $\alpha_1$ is solid, $\alpha_2$ is dashed}
\label{fig:pf for a3}
\end{figure}

Another combinatorial approach to distinguished bases  was proposed in \cite{buan}, where they were related to the maximal chains of non-crossing partitions. Stanley in \cite{stanley} constructed a bijection between the maximal chains of non-crossing partitions and parking functions. Both bijections appear to be quite nontrivial.

\begin{theorem}
The ``initial vector'' map coincides with the composition of bijections from \cite{buan} and \cite{stanley}.
\end{theorem}

In Section \ref{sec:quiver}, we discuss the relation of the above constructions to quiver representations.
A theorem of Gabriel \cite{Gab} identifies the positive roots of the $A_n$ root system with the indecomposable representations of the corresponding quiver; the Seifert form corresponds to the Euler form on the Grothendieck group $K_0(\Rep A_n)$. We show that distinguished bases (made of positive roots) enumerate collections $(E_1,\ldots,E_n)$ of indecomposable quiver representations such that
\begin{equation}
\label{exts}
\Ext^{k}(E_i,E_j)=0\quad \forall i>j,k\ge 0.
\end{equation}
Such are called  exceptional sequences, or  exceptional collections (e.g. \cite{araya,buan,buan2}). They were first investigated in works \cite{GoRu} and \cite{Ru}. The set of exceptional sequences carries as well an action of the braid group (see, e.g., \cite{buan2,CB,R}) which corresponds naturally to its action on the set of distinguished bases in our case.
In terms of quivers our main theorem says that there exists a filtration on $K_0(\Rep  A_n)$ such that exceptional sequences can be uniquely reconstructed from the filtration levels of their components.

\section*{Acknowledgements}

E. G. would like to thank S. Gusein-Zade, A. Kirillov Jr., P. Seidel and D. Orlov for useful discussions.
M. G. is grateful to B. Keller for the support and remarks on the preliminary version.
The research of E. G. was partially supported by the grants  RFBR-10-01-00678, NSh-8462.2010.1  and the Dynasty fellowship for young scientists. The research of M.~G. was supported by Fondation Sciences math\'{e}matiques de Paris and by R\'{e}seau de Recherche Doctoral en Math\'{e}matiques de l'\^{I}le de France.

\section{Distinguished bases}

Let $e_1,\ldots, e_n$ denote the positive simple roots of the $A_n$ root system.
We will denote positive roots as
$e_{ij}:=e_i+e_{i+1}+\ldots+e_j.$

\begin{definition}
The {\it Seifert form} on the $A_n$ root system is a non-symmetric bilinear form defined by the equations
$\left\langle e_i,e_i \right\rangle =1,\ \left\langle e_i,e_{i+1}\right\rangle =-1,$ and
$\left\langle e_i,e_j \right\rangle =0,\  j\neq i,i+1.$
\end{definition}

\begin{proposition}
For all $x,y,$ one has $\left\langle x,y\right\rangle+\left\langle y,x\right\rangle=(x,y),$
where the right hand side is the standard (Cartan) scalar product of $x$ and $y$.
\end{proposition}

\begin{corollary}
If $v$ is a root of the $A_n$ root system, then
$\left\langle v,v\right\rangle=1.$
\end{corollary}

\begin{definition}
A basis $\{a_1,\ldots,a_n\}$ of the $A_n$ root system is  called {\it distinguished}, if $a_j$ are roots in the root system and $\left\langle a_j,a_i\right\rangle=0,$ for $j>i$.
This means that the matrix of the Seifert form is upper-triangular in the distinguished basis.
\end{definition}

Remark that the change of sign for a root transforms a distinguished basis to a distinguished one.
Therefore, from now on, we will consider the distinguished bases made of positive roots only.
The following lemma describes the values of the Seifert form on pairs of positive roots and can be easily proved by case by case analysis.

\begin{lemma}
\label{L1}
The value of the Seifert form on a pair of positive roots can be computed in a following way:
$$\left\langle e_{ij},e_{kl}\right\rangle=\begin{cases}
1,\quad \mbox{\rm if}\quad k\le i\le l\le j\\
-1,\quad  \mbox{\rm if}\quad i\le k-1\le j<l\\
0,\quad \mbox{\rm otherwise} \\
\end{cases}
$$
\end{lemma}

%\begin{proof}
%We can prove it case by case.

%1) $i\le k-1< l\le j$. Then
%$$\left\langle e_{i}+\ldots+e_{j},e_{k}+\ldots+e_{l}\right\rangle=\sum_{s=k-1}^{l-1}\left\langle e_s,e_{s+1}\right\rangle+\sum_{s=k}^{l}\left\langle e_s,e_s\right\rangle=0;$$

%2) $i\le k-1\le j<l$. Then
%$$\left\langle e_{i}+\ldots+e_{j},e_{k}+\ldots+e_{l}\right\rangle=\sum_{s=k-1}^{j}\left\langle e_s,e_{s+1}\right\rangle+\sum_{s=k}^{j}\left\langle e_s,e_s\right\rangle=-1;$$

%3) $i\le j<k-1<l$. Then 
%$$\left\langle e_{i}+\ldots+e_{j},e_{k}+\ldots+e_{l}\right\rangle=0.$$

%4) $k\le i\le j< l$. Then
%$$\left\langle e_{i}+\ldots+e_{j},e_{k}+\ldots+e_{l}\right\rangle=\sum_{s=i}^{j}\left\langle %e_s,e_{s+1}\right\rangle+\sum_{s=i}^{j}\left\langle e_s,e_s\right\rangle=0.$$

%5) $k\le i\le l\le j$. Then
%$$\left\langle e_{i}+\ldots+e_{j},e_{k}+\ldots+e_{l}\right\rangle=\sum_{s=i}^{l-1}\left\langle e_s,e_{s+1}\right\rangle+\sum_{s=i}^{l}\left\langle e_s,e_s\right\rangle=1.$$

%6) $k\le l< i\le j$. Then 
%$$\left\langle e_{i}+\ldots+e_{j},e_{k}+\ldots+e_{l}\right\rangle=0.$$

%\end{proof}

\begin{definition}
The support of a root $a = e_{ij}$ is the set $\supp(a) = \left\{i,i+1,\ldots,j\right\}.$
\end{definition}

\begin{corollary}
\label{C2}
In a distinguished basis, for every two roots, either the support of one is contained in the support of the other or these supports do not intersect.
\end{corollary}

\begin{proof}
Let $a=e_i+\ldots+e_j,\quad b=e_k+\ldots+e_l,\quad i<k<j<l.$ By Lemma \ref{L1},  $\left\langle a,b\right\rangle=-1$ and  $ \left\langle b,a\right\rangle=1$. On the other hand, in the distinguished basis the value of the Seifert form on them should vanish in one of the orders.
\end{proof}

Let us associate with each positive root $e_{ij}$ an arc  above the real axis and with ends $(i - 1, j)$ on it.  

\begin{proposition} \label{dbarcs}
Distinguished bases for root system $A_n$ correspond bijectively to ordered collections of $n$ pairwise non-intersecting arcs with the following properties:
\begin{itemize}
\item[1)] If two arcs have same left ends, then the inside arc has bigger label,
%arc inside goes after the arc outside.

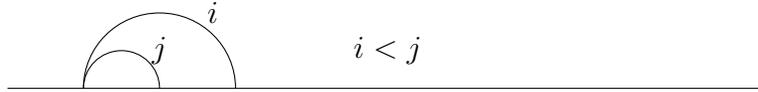
\begin{tikzpicture}
\draw (3,0) arc (0:180:1cm);
\draw (2,0) arc (0:180:0.5cm);
\draw[-,>=stealth] (0,0)--(10,0);
\draw (2.7,1) node {$i$};
\draw (2,0.5) node {$j$};
\draw (5,0.5) node {$i<j$};
\end{tikzpicture}

\item[2)] If two arcs have same right ends, then the inside arc has smaller label,
%arc outside goes after the arc inside.

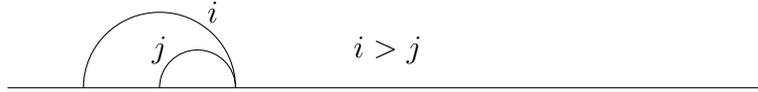
\begin{tikzpicture}
\draw (3,0) arc (0:180:1cm);
\draw (3,0) arc (0:180:0.5cm);
\draw[-,>=stealth] (0,0)--(10,0);
\draw (2.7,1) node {$i$};
\draw (2,0.5) node {$j$};
\draw (5,0.5) node {$i>j$};
\end{tikzpicture}

\item[3)] If right end of arc $i$ coincides with left end of arc $j$, then $i<j,$

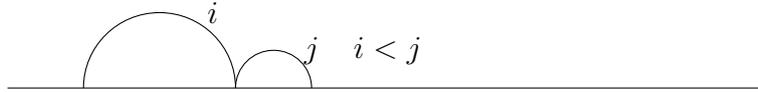
\begin{tikzpicture}
\draw (3,0) arc (0:180:1cm);
\draw (4,0) arc (0:180:0.5cm);
\draw[-,>=stealth] (0,0)--(10,0);
\draw (2.7,1) node {$i$};
\draw (4,0.5) node {$j$};
\draw (5,0.5) node {$i<j$};
\end{tikzpicture}

\item[4)] The arcs form a graph without cycles.
%There is no ``chain'' in this collection, i.e. subcollection of arcs that left end of the next arc coincides with right end of the previous one, that there is an arc (in this collection) whose left end coincides with left end of the first arc in chain and right end coincides with right end of the last arc in chain.

\end{itemize}
\end{proposition}

\begin{proof}
The property (4) ensures that we have $n$ linearly independent positive roots, i.e. a basis $A = \left\{a_1,\ldots,a_n\right\}$ of positive roots. The properties (1)-(3),  by Lemma \ref{L1} 
and Corollary \ref{C2}, imply that this basis is distinguished.
\end{proof}

\begin{remark}
Similar pictures appeared in \cite{araya}, where, however, the order of arcs was not considered.
\end{remark}

\begin{remark}
It is easy to see that the property (4) follows from (1)-(3): since the arcs do not intersect each other, 
a minimal cycle may consist only of one big arc containing small arcs (see Figure \ref{fig:cycle}). Suppose that the small arcs have labels $i_1,\ldots,i_k$, and the big one has label $j$. By (3), we have $i_1<\ldots<i_k$; by (2), $j>i_{k},$ and by (1) $j<i_{1}$. Contradiction.
\end{remark}

\begin{figure}[ht]
\begin{tikzpicture}
\draw (8,0) arc (0:180:3.5cm);
\draw (2,0) arc (0:180:0.5cm);
\draw (3,0) arc (0:180:0.5cm);
\draw (5,0) arc (0:180:1cm);
\draw (6,0) arc (0:180:0.5cm);
\draw (8,0) arc (0:180:1cm);
\draw[-,>=stealth] (0,0)--(10,0);
\end{tikzpicture}
\caption{}
\label{fig:cycle}
\end{figure}

\begin{definition}
We introduce a partial order on roots: $a\succeq b$ if $\supp(a) \supset \supp(b).$ 
Consider a root $a_i$ in a distinguished basis $A = (a_1,\ldots,a_n).$ Define 
$$
\Span(a_i, A):=\bigcup \limits_{a_i\succ a_j} \supp(a_j),\ 
\Gap(a_i, A) := \supp(a_i) \backslash \Span(a_i, A).
$$
\end{definition}

%We will need two following lemmas in the proof of the main theorem.

\begin{lemma} \label{hole}
The set $\Gap(a_i,A)$ consists of a single element.
\end{lemma}

\begin{proof}
If $\Gap(a_i, A)=\emptyset$, then the sum of $\succ$-maximal elements $a_j\prec a_i$ equals to $a_i$.
If $\Gap(a_i, A)$ contains two simple roots $e_j, e_k$, then the support of every root from $A$ contains either 
both $e_j$ and $e_k$ or none of them. 
In both cases, $A$ is linearly dependent. Contradiction. 
\end{proof}

%\begin{lemma} \label{span}
%Suppose that $a_k= e_{i} + \ldots + e_{j}$. Then
%$$|\{s|a_s\prec a_k\}|=j-i.$$
%\end{lemma}

%\begin{proof}
%Consider a root system of type $A_{j-i+1}$ generated by simple roots $e_{i}, e_{i+1}, \ldots, e_{j}.$  
%The set $$\{a_k\}\cup \{s|a_s\prec a_k\}|$$ is a basis for this root system, so it should contain $j-i+1$ elements.
%\end{proof}

\begin{lemma} \label{chain}
Suppose that $a_i=e_{\alpha}+\ldots+e_{\beta},$ and 
$e_{s}\in \Span(a_i, A)$. Then there is a sequence $k_1,k_2,\ldots, k_r$ such that either
\begin{equation}
\label{left}
k_j>i\quad \forall j,\qquad [\alpha,s]\subset \bigcup_{j}\supp(a_{k_j})
\end{equation}
or 
\begin{equation}
\label{right}
k_j<i\quad \forall j,\qquad [s,\beta]\subset \bigcup_{j}\supp(a_{k_j}).
\end{equation}
\end{lemma}

\begin{proof}
Let $c$ be the maximal number such that there exists a sequence $\{k_j\}$
satisfying (\ref{left}). By Lemma \ref{dbarcs}, $c+1\notin \Span(a_i, A)$ -- otherwise a root starting with $c+1$ should go after the root ending with $c$ in the basis $A$, and $c$ is not maximal. Therefore, $c+1=\Gap(a_i, A),$ and for all $s<\Gap(a_i, A)$ there exists a sequence $\{k_j\}$ satisfying (\ref{left}).
Analogously, for all $s>\Gap(a_i, A)$ there exists a sequence $\{k_j\}$ satisfying (\ref{right}).
\end{proof}

A theorem of Lyashko \cite{lya} and Looijenga \cite{loo} states that the root system of type $A_n$ has $(n+1)^{n-1}$ distinguished bases of positive roots.
Let us give a recursive construction of these bases.

\begin{definition}
The {\it right orthogonal complement} to a subspace $V$ is defined as
$$V^{\bot}_{r}=\{y\ |\ \forall x\in V\  \left\langle x,y\right\rangle=0\}.$$
\end{definition}

\begin{lemma}
\label{orth}
The right orthogonal complement to a root $e_{k}+\ldots+e_{m}$ %with the restriction of the Seifert form
is isomorphic to $A_{m-k}\oplus A_{n-m+k-1}.$
\end{lemma}

\begin{proof}
Using Lemma \ref{L1}, one can present an explicit basis in this orthogonal complement:
$$\left\langle e_{k}+\ldots+e_{m}\right\rangle^{\bot}_{r}=\left\langle e_1,\ldots,e_{k-2},e_{k-1}+\ldots+e_{m},e_{m+1},\ldots e_{n}\right\rangle\oplus\left\langle e_{k+1},\ldots,e_{m}\right\rangle.$$
It is clear that the first subspace has type $A_{n-m+k-1}$, the second one has type $A_{m-k},$ and they are orthogonal to each other.
\end{proof}

%\begin{lemma}
%Let $D(n)$ be the number of the distinguished bases in $A_n$. Then the following recursion relation holds:
%$$D(n)=\sum_{l=0}^{n-1}(n-l)\binom{n-1}{l}D(l)D(n-1-l).$$   
%\end{lemma}

%\begin{proof}
It is well known (e.g \cite{book}) that for every root of the root system one can find a distinguished basis starting from this root.
Therefore, we can choose arbitrary root $e_{k}+\ldots+e_{m}$ as a first vector of a  distinguished basis, and all other vectors from this basis will belong to its right orthogonal complement. If $m-k=l$ is fixed, then we have $n-l$ options to choose a root, and by Lemma  \ref{orth} the complement is split as $A_{l}\oplus A_{n-1-l}$. Therefore, we have to choose bases in $A_{l}$ and $A_{n-1-l}$ and then shuffle them in one of $\binom{n-1}{l}$ ways.
%\end{proof}

\begin{example}
We can describe all distinguished bases of positive roots for $A_1$, $A_2$ and $A_3$.
$$
A_1:\quad \{e_1\}
$$
$$
A_2:\quad \{e_1,e_2\}\quad \{e_2,e_1+e_2\}\quad \{e_1+e_2,e_1\}
$$
$$
A_3: \{e_1,e_2,e_3\}\quad \{e_1,e_3,e_2+e_3\}\quad \{e_1,e_2+e_3,e_2\}$$
$$\{e_2,e_1+e_2,e_3\}\quad \{e_2,e_3,e_1+e_2+e_3\}\quad \{e_2,e_1+e_2+e_3,e_1+e_2\}$$
$$\{e_3,e_1,e_2+e_3\}\quad \{e_3,e_2+e_3,e_1+e_2+e_3\}\quad \{e_3,e_1+e_2+e_3,e_1\}$$
$$\{e_1+e_2,e_1,e_3\}\quad \{e_1+e_2,e_3,e_1\}$$
$$\{e_2+e_3,e_1+e_2+e_3,e_2\}\quad \{e_2+e_3,e_2,e_1+e_2+e_3\}$$ 
$$\{e_1+e_2+e_3,e_1,e_2\}\quad  \{e_1+e_2+e_3,e_2,e_1+e_2\}\quad \{e_1+e_2+e_3,e_1+e_2,e_1\}$$
\end{example}

\section{Parking functions}

\begin{definition} 
A {\it parking,} or {\it preference function} on $n$ elements is a function
$$f : \left\{1, \ldots, n\right\} \rightarrow \left\{1, \ldots, n\right\}
\ \mbox{\rm such that}\ |f^{-1}(\left\{1, \ldots, k\right\})| \geq k\ \forall k.$$
\end{definition}

\begin{definition}
Let $e=e_{i}+e_{i+1}+\ldots+e_{j}$ be a root of $A_n$ root system.
We define its {\it initial point} as $\In(e)=i$.
For a basis $A=\{a_1,\ldots, a_n\},$ we define the {\it initial vector}
as the integer sequence
$\In(A)=(\In(a_1),\ldots,\In(a_n)).$
\end{definition}

\begin{lemma}
If  $A$ is a basis in the root system, then $\In(A)$ is a parking function.
\end{lemma}

\begin{proof}
Suppose that $\In(A)$ is not a parking function, i. e. there exists $k$ such that
$$|f^{-1}(\{1,\ldots,k\})|<k.$$ Therefore, $|f^{-1}(\{k+1,\ldots,n\})|>n-k$,
and  $A$ contains more than $n-k$ vectors from the $(n-k)$-dimensional subspace
spanned by $e_{k+1},\ldots e_{n}$, hence $A$ is a linearly dependent collection.
\end{proof}

\begin{theorem}
\label{Tmain}
Distinguished bases for $A_n$ root system are in 1-to-1 correspondence with the parking functions on $n$ elements.
The bijection is given by the ``initial vector'' map.
\end{theorem}

\begin{proof}
We construct an inverse to the ``initial vector'' map.
Given a parking function $f$ on $n$ elements, we need to find a set of $n$ positive roots $A = \left\{a_1,\ldots,a_n\right\}$ such that $\In(A) = f.$ 
We will describe these roots %(or, equivalently, $b(k)$) in a reverse lexicographic order, namely, 
%in descending order of $f(k),$ and for $a_k$ and $a_l$ with $f(k) = f(l), k < l$ we will construct $a_l$ first.
in the following order: if $f(k)<f(l)$ or $f(k) = f(l), k < l$, we will list $a_l$ first.

Let us recover the root $a_k$. Remark that, by construction, all roots $a_{s}\prec a_k$ are already found. 
Let $C(k)$ be the union of supports of constructed roots $a_i$ such that $i > k;$ $B(k)$ be  the union of supports of constructed roots $a_j$ such that $j < k.$ Consider the numbers

$$c(k) = 
\left\{\begin{matrix}
\max (X = \left\{j: [f(k),j] \subset C(k) \right\}), & X \neq \emptyset;
\\ f(k) - 1,  & X = \emptyset;\end{matrix}\right.$$

$$b(k) = 
\left\{\begin{matrix}
\max (Y = \left\{j: [c(k)+2,j] \subset B(k) \right\}), & Y \neq \emptyset;
\\ c(k) + 1,  & Y = \emptyset,\end{matrix}\right.$$
and the root $a_k=e_{f(k)} + \ldots + e_{b(k)}.$
Since $\supp(a_k) \backslash \left(B(k) \cup  C(k) \right) = \left\{c(k) + 1\right\} \neq \emptyset$,
constructed roots are linearly independent. 
By Lemma \ref{L1} one has: %and the choice of $X, Y, c(k)$ and $b(k)$ that for already found $a_j$ 
$$
\left\langle a_k,a_j\right\rangle = 0, \quad k > j;\quad 
 \left\langle a_j, a_k\right\rangle = 0, \quad j > k.
$$
This proves  that we obtain a distinguished basis. Since $\In(A) = f$ by construction, $\In$ is surjective.
To prove that $\In$ is injective, we have to check that $a_k$ is uniquely determined by the previously constructed roots.
By construction, $\left(B(k)\cup C(k)\right) \cap \supp(a_k)=\Span(a_k,A).$
By Lemma \ref{hole}, the set $\Gap(a_k, A) = (\supp(a_k) \backslash \Span(a_k, A))$ contains exactly 1 element,
which divides $\Span(a_k, A)$ in two connected components. If $a_k=e_{f(k)}+\ldots+e_{j}$, then $e_{j+1}\notin B(k)$ - otherwise
the vector starting from $e_{j+1}$ should go after $a_k$ in $A$.  It rests to prove that all numbers between $f(k)$ and $\Gap(a_k, A)$ are covered by $C(k)$, and all numbers between $\Gap(a_k, A)$ and $j$ are covered by $B(k)$. This follows from Lemma \ref{chain}.
\end{proof}

\begin{example}
Let us use this 
method to construct the distinguished bases for the parking functions $2,2,1$ and $2,1,1$.

$2,2,1$: we will start with %the second vector, 
$a_2$, then %find the first one, then the third one. 
construct $a_1$ and $a_3$. $C(2) = B(2) = \emptyset,$ hence $c(2) = 1, b(2) = 2, a_2 = e_2.$ Thus $C(1) = \left\{2\right\}, B(1) = \emptyset,$ and $c(1) = 2 \Rightarrow b(1) = 3 \Rightarrow a_1 = e_2 + e_3.$ Finally, $C(3) = \emptyset, B(3) = \left\{2, 3\right\};$ therefore, $c(3) =0 \Rightarrow b(3) = 3 \Rightarrow a_3 = e_1 + e_2 + e_3.$ %To summarize, desired distinguished basis is $\left\{e_2 + e_3, e_2, e_1 + e_2 + e_3\right\}.$

$2,1,1$: we will start with %the first vector, 
$a_1$, then %find the third one, then the second one. 
construct $a_3$ and $a_2$. $C(1) = B(1) = \emptyset,$ hence $c(1) = 1, b(1) = 2, a_1 = e_2.$ Thus $C(3) = \emptyset, B(3) = \left\{2\right\},$ and $c(3) = 0 \Rightarrow b(3) = 2 \Rightarrow a_1 = e_1 + e_2.$ Finally, $C(2) = \left\{1, 2\right\}, B(2) = \left\{2\right\};$ therefore, $c(2) =2 \Rightarrow b(3) = 3 \Rightarrow a_3 = e_1 + e_2 + e_3.$ %To summarize, desired distinguished basis is $\left\{e_2, e_1 + e_2 + e_3, e_1 + e_2\right\}.$
\end{example}

It turns out that the reconstruction of a distinguished basis from a parking function can be drawn on a picture. We will need some combinatorial constructions from \cite{haglund}.

\begin{definition} 
Let us define $\mathbb{Y}_n$ as the set of all Young diagrams inside triangle formed by the coordinate axis and the line $y = x - n.$ %In other words, $\mathbb{Y}_n$ is a set of all Young diagrams in IV quadrant, bounded by a Dyck path of length $2n,$ consisting of vectors $(0,1)$ and $(1,0).$ 
%Such diagrams are  bounded by  Dyck paths of length $2n$ with steps $(0,1)$ and $(1,0).$
The boundary of such a diagram is a lattice path of length $2n$ with steps $(0,1)$ and $(1,0),$
which we will call the {\it Dyck path}.
\end{definition}

It is well known that the number of elements in $\mathbb{Y}_n$ equals to the $n$-th Catalan number 
$$
c_n = \frac{1}{n+1}\binom{2n}{n}.
$$
 
\begin{definition}
A {\it parking function diagram} is a diagram from $\mathbb{Y}_n$, where numbers from $1$ to $n$ are written at the end of each row
such that in every column the numbers are increasing upwards.
We denote by $P_k$ the SE angle of a row with number $k.$
\end{definition}

Given a parking function diagram $D$, consider a function $f_{D}:\{1,\ldots,n\}\to \{1,\ldots,n\}$ mapping a number to its $x$-coordinate on $D$ increased by 1. One can check (see e. g. \cite{haglund}) that $f_{D}$ is a parking function and the correspondence between $D$ and $f_{D}$ is bijective.

\begin{example}
The parking function diagram $D$ on Figure \ref{fig:example of pf} corresponds to the function 
$$f_{D}=\left(\begin{matrix} 1& 2& 3& 4& 5\\ 1& 5& 3& 1& 4\end{matrix}\right).$$

\begin{figure}[ht]
\begin{tikzpicture}[line width=0.6pt]
  \draw [<->,thick] (1,4) node (yaxis) [above] {$y$}
     |- (5,3) node (xaxis) [right] {$x$};
  \draw [thick] (1,0) -- (1,3);
  \draw [thick] (0,3) -- (1,3);
  \draw (3,3) -- (3,2.5) -- (2.5,2.5) -- (2.5,2) -- (2,2) -- (2,1.5) -- (1,1.5);
  \draw (1.5,3) -- (1.5,1.5);
  \draw (2,3) --(2,2);
  \draw (2.5,3) -- (2.5,2.5);
  \draw (1,2.5) -- (2.5,2.5);
  \draw (1,2) -- (2.5,2);
  \draw (1,2.5) -- (1,2);
  \draw (2.5,2.5) -- (2.5,2);
  \draw [dashed] (3.5,3) -- node [below right] {$y = x - 5$}(1,0.5);
  \draw (1.15,0.9) node {\bf 1};
  \draw (1.15,1.3) node {\bf 4};
  \draw (2.1,1.8) node {\bf 3};
  \fill[black] (2,1.5) circle (2pt);
  \draw (2,1.5) node [below] {$\bf P_3$};  
  \draw (2.6,2.3) node {\bf 5};
  \draw (3.1,2.8) node {\bf 2};
\end{tikzpicture}
\caption{}
\label{fig:example of pf}
\end{figure}
\end{example}

%\begin{proof}
%We have to check that for every $k$ 
%$$|f_{D}^{-1}(\{1,\ldots,k\}|\ge k.$$
%This is true because there are at least $k$ slots for the numbers $x$ in $D$ %to the left of the line $x=k$.
%with $x<k$.

%Let $f$ be a parking function, let us construct the corresponding diagram.
%Let $$M_{f}(x)=|\{y|f(y)<f(x)\}|,\quad N_{f}(x)=|\{y|f(y)=f(x),y<x\}|.$$
%Let us put the number $x$ in the cell $(f(x), M_{f}(x)+N_{f}(x)-n)$. One can check that these numbers bound a diagram %$D\in \mathbb{Y}_n$,
%and $f_{D}=f$.
%\end{proof}

\begin{theorem}
\label{T2}
Let $A=(a_1,\ldots,a_n)$ be a distinguished basis for $A_{n}$ root system, let $D(A)$ be the parking function diagram corresponding to $\In(A)$. 

Let us start from $P_k$ and go strictly north-east until we meet first $P_l$ such that $l > k$ or encounter the Dyck path or the $x$-axis.
Our path started at $x$-coordinate $\In(a_k) - 1$ and ended at some $x$-coordinate $\Ter(k)$. Then
$$a_{k}=e_{\In(a_k)}+\ldots+e_{\Ter(k)}.$$
\end{theorem}

\begin{proof}
Let us apply this procedure for every $k$ and consider roots $a'_{k}=e_{\In(a_k)}+\ldots+e_{\Ter(k)}.$
By Theorem \ref{Tmain}, it is sufficient to prove that $a'_{k}$ form a distinguished basis. 

For this purpose we need to check the conditions (1)-(3) of Proposition \ref{dbarcs}.
The condition (1) holds, since the numbers are increasing in columns of $D(A).$
If two arcs have the same end, then the corresponding numbers lie on the same diagonal line in $D(A)$,
hence the left one is bigger and the condition (2) is satisfied as well.
Finally, if we are in situation (3), then the corresponding numbers lie on the same diagonal line in $D(A)$,
but the left one is smaller.
\end{proof}

%\begin{remark}
%Fix a distinguished basis and corresponding parking function diagram. Note that projections of associated arcs from %Proposition \ref{dbarcs} on the $x$-axis coincide with projections of intervals defining roots in Theorem \ref{T2}
%\end{remark}

\begin{example}
Let $n=12$ and the parking function is given by the formula
$$f=\left(\begin{matrix} 1& 2& 3& 4& 5& 6& 7& 8& 9& 10\quad 11\quad 12\\ 3& 11& 7& 5& 9& 8& 5& 2& 1& 10\quad 2\quad 12\end{matrix}\right).$$
The corresponding diagram is shown in Figure \ref{fig:example of bases}.
The corresponding distinguished basis has a form
$$
\left\{e_3, e_{11},e_7, e_{5,7}, e_9, e_{8,9}, e_5, e_{2,9}, e_{1,9}, e_{10,11}, e_{2,3}, e_{12}\right\}.
$$
It can be illustrated by arcs as in Figure \ref{fig:example of arcs}.
\end{example}

%Parking function 
%$$x: \quad\quad1\quad2\quad3\quad4\quad5\quad6\quad7\quad8\quad9\quad10\quad11\quad12$$
%$$f(x): \quad3\quad11\quad7\quad5\quad9\quad8\quad5\quad2\quad1\quad10\quad2\quad12
%$$

\begin{figure}[ht]
\begin{tikzpicture}[scale=0.6, line width=0.6pt]
  \draw [<->,thick] (1,15) node (yaxis) [above] {$y$}
     |- (16,14) node (xaxis) [right] {$x$};
  \draw [thick] (1,0) -- (1,14);
  \draw [thick] (0,14) -- (1,14);
  \draw (1,13) -- (12,13) --node [right]{\bf 12} (12,14);
  \draw (1,12) -- (11,12) --node [right]{\bf 2}(11,13) -- (11,14);
  \draw (1,11) -- (10,11) --node[right]{\bf 10}(10,12) -- (10,14);
  \draw (1,10) -- (9,10) --node[right]{\bf 5}(9,11) -- (9,14);
  \draw (1,9) -- (8,9) --node[right]{\bf 6}(8,10) -- (8,14);
  \draw (1,8) -- (7,8) --node[right]{\bf 3}(7,9) -- (7,14);
  \draw (1,7) -- (5,7) --node[right]{\bf 7}(5,8) -- (5,14);
  \draw (1,6) -- (5,6) --node[right]{\bf 4}(5,7) -- (5,14);
  \draw (1,5) -- (3,5) --node[right]{\bf 1}(3,6) -- (3,14);
  \draw (1,4) -- (2,4) --node[right]{\bf 11}(2,5) -- (2,14);
  \draw (1,3) -- (2,3) --node[right]{\bf 8}(2,4) -- (2,14);
  \draw (1,2) --node[right]{\bf 9}(1,3);
  \fill[black] (7,8) circle (4pt);
  \draw (7,8) node [below] {{\bf $P_3$}};
  \fill[black] (5,6) circle (4pt);
  \draw (5,6) node [below] {{\bf $P_4$}};
  \draw (6,8) -- (6,14);
  \draw (4,6) -- (4,14);
  %\draw [dashed] (1,2) -- node [below right] {$y = x - 12$}(13,14);
\end{tikzpicture}
\caption{}
\label{fig:example of bases}
\end{figure}

%Permutation: $(9\quad8\quad11\quad1\quad4\quad7\quad3\quad6\quad5\quad10\quad2\quad12).$

\begin{figure}[ht]
\begin{tikzpicture}[scale = 0.8]
\draw (9,0) arc (0:180:0.5cm);
\draw (8,0) arc (0:180:0.5cm);
\draw (8,0) arc (0:180:1cm);
\draw (6,0) arc (0:180:0.5cm);
\draw (6,0) arc (0:180:1cm);
\draw (6,0) arc (0:180:4cm);
\draw (6,0) arc (0:180:4.5cm);
\draw (4,0) arc (0:180:0.5cm);
\draw (4,0) arc (0:180:1.5cm);
\draw (2,0) arc (0:180:0.5cm);
\draw (0,0) arc (0:180:0.5cm);
\draw (0,0) arc (0:180:1cm);
\draw[-,>=stealth] (-4,0)--(10,0);
\draw (-1,0.5) node {1};
\draw (7,0.5) node {2};
\draw (3,0.5) node {3};
\draw (2.5,1.7) node {4};
\draw (5,0.5) node {5};
\draw (4.8,1.2) node {6};
\draw (2,0.5) node {7};
\draw (2.3,3.7) node {8};
\draw (0,4.5) node {9};
\draw (7,1.3) node {10};
\draw (-1,1.3) node {11};
\draw (9,0.5) node {12};
\draw (-3,-0.3) node {$0$};
\draw (-2,-0.3) node {$1$};
\draw (-1,-0.3) node {$2$};
\draw (0,-0.3) node {$3$};
\draw (1,-0.3) node {$4$};
\draw (2,-0.3) node {$5$};
\draw (3,-0.3) node {$6$};
\draw (4,-0.3) node {$7$};
\draw (5,-0.3) node {$8$};
\draw (6,-0.3) node {$9$};
\draw (7,-0.3) node {$10$};
\draw (8,-0.3) node {$11$};
\draw (9,-0.3) node {$12$};
\end{tikzpicture}
\caption{}
\label{fig:example of arcs}
\end{figure}

\begin{proposition}
Let $\sigma\in S_{n}$ be a permutation. There exists a distinguished basis $A$ with $\In(A)=\sigma$, which can be constructed 
using the following procedure.
Let $\Ter(k)$ be the maximal number such that
$$[\sigma(k),\Ter(k)]\subset \sigma(\{1,\ldots,k\}).$$ Then $a_{k}=e_{\sigma(k)}+\ldots+e_{\Ter(k)}$.
\end{proposition}

\begin{proof}
A permutation is a parking function whose diagram is a maximal Young diagram in $\mathbb{Y}_n$.
Now the statement follows from Theorem \ref{T2}.
\end{proof}

\begin{proposition}
A parking function $f$ is non-decreasing if and only if the numbers on its diagram are equal to their $y$-coordinates (shifted by $n$).
Therefore, 
\begin{itemize}
\item[(i)]The procedure of Theorem \ref{T2} can be described as follows: we start from the point $P_k$ and go strictly north-east
until we touch a diagram or the $x$-axis. If $(\In(k)-1)$ and $\Ter(k)$ are the $x$-coordinates of our start and finish respectively, then
$$a_{k}=e_{\In(k)}+\ldots+e_{\Ter(k)}.$$
\item[(ii)] Non-decreasing parking functions on $n$ elements are in 1-to-1 correspondence with Dyck paths of the length $2n.$
\end{itemize}
\end{proposition}

%\begin{proof}
%This is obvious.
%\end{proof}
%\begin{remark}
%This procedure was constructed in \cite{Gor} where it was used to describe mutations in the cluster algebra of type $A_n$. 
%\end{remark}

\begin{example}
Consider a non-decreasing parking function with the values $(1,1,2,2,2,4,6)$. Its diagram is shown in Figure \ref{fig:non-decreasing pf} and one can check that it corresponds to the distinguished basis  
$A = (e_{17}, e_1, e_{25}, e_{23}, e_2, e_4, e_6).$
\end{example}

\begin{figure}[ht]
\begin{tikzpicture}[line width=0.6pt]
  \draw [<->,thick] (1,4) node (yaxis) [above] {$y$}
     |- (5,3) node (xaxis) [right] {$x$};
  \draw [thick] (1,-1) -- (1,3);
  \draw [thick] (0,3) -- (1,3);
  \draw (3.5,3) -- (3.5,2.5) -- (2.5,2.5) -- (2.5,2) -- (1.5,2) -- (1.5,1) -- (1,1);
  \draw (1.5,3) -- (1.5,1);
  \draw (2,3) --(2,2);
  \draw (2.5,3) -- (2.5,2);
  \draw (3,3) -- (3,2.5);
  \draw (1,2.5) -- (2.5,2.5);
  \draw (1,2) -- (2,2);
  \draw (1,0.5) -- (1.5,0.5) -- (1.5,1);
  \draw (1,1.5) -- (1.5,1.5);
  \draw [dashed] (4.5,3) -- node [below right] {$y = x - 7$}(1,-0.5);
  \draw [dashed] (1.5,1) -- (3,2.5);
  %\draw [red] (1.5,1) -- (1.5,2.5) -- (3,2.5);
  \draw [dashed] (1,0) -- (4,3);  
  \draw [dashed] (1.5,1.5) -- (2,2);  
  %\fill[blue] (1.5,0.5) circle (2pt);
  %\fill[green] (3.5,2.5) circle (2pt);
  \draw (1.2,-0.1) node {\bf 1};
  \draw (1.2,0.3) node {\bf 2};
  \draw (1.7,0.8) node {\bf 3};
  \draw (1.7,1.3) node {\bf 4};
  \draw (1.7,1.8) node {\bf 5};
  \draw (2.7,2.3) node {\bf 6};
  \draw (3.7,2.8) node {\bf 7};
\end{tikzpicture}
\caption{}
\label{fig:non-decreasing pf}
\end{figure}

%Let us find $a_3$ directly by $\Lambda.$ $in(a_3) = \lambda_{7-3+1} + 1=\lambda_5 + 1 = 2.$ To find $\nu_5,$ start from the point $(1,-5)$ and go along the line 
%$$x - y = \lambda_5 + 5 \Leftrightarrow x - y = 6$$
%right and upwards. We touch the diagram again firstly at the point $(5,-1),$ thus $\nu_5 = 5.$ It is equal to $g(3),$ as we claimed.

%We also draw a red triangle inside which there lay intersections of rows corresponding to elements of $A$ spanned by $a_3$ with a quadrat $\left\{(x,y)|in(3) \leq x \leq g(3)-1,1-(7-3+1)=-4 \leq %y \leq -4+g(3)-in(3) \right\}.$

%\end{example}

\section{Braid group action}
\label{sec:braid}

\begin{definition}(\cite{book})
Let $A=(a_1,\ldots,a_n)$ be some $n$-tuple of roots. We define two operations:
$$\alpha_{k}(A)=(a_1,\ldots,a_{k-1},a_{k+1},a_k-\left\langle a_k,a_{k+1}\right\rangle a_{k+1},a_{k+2},\ldots,a_n),$$
$$\beta_{k}(A)=(a_1,\ldots,a_{k-1},a_{k+1}-\left\langle a_k,a_{k+1}\right\rangle a_{k},a_k,a_{k+2},\ldots,a_n).$$
We will call them {\it left} and {\it right mutations} respectively. % (check signs/indices!!)
\end{definition}

\begin{proposition}(\cite{book})
If $A$ is a distinguished basis, then $\alpha_k(A)$ and $\beta_{k}(A)$ are distinguished bases too.
The following relations hold:
$$
\beta_k=\alpha_{k}^{-1},\quad \alpha_{k}\alpha_{m}=\alpha_{m}\alpha_{k}\quad (|m-k|>1),\quad \alpha_{k}\alpha_{k+1}\alpha_{k}=\alpha_{k+1}\alpha_{k}\alpha_{k+1}.
$$ 
In other words, the $\alpha_{k}$ define a representation of a braid group of type $A_n$ with $n$ strands on the set of distinguished bases.
\end{proposition}

%Let us describe the induced braid group action on parking functions. 

\begin{lemma} \label{arcsmut}
Suppose that $a$ and $b = b_{\In(b)}+\ldots+b_{\Ter(b)}$ are two positive roots such that $\left\langle b,a\right\rangle=0$. Let $c=a-\left\langle a,b\right\rangle b$. Then $c=a$ if $\left\langle a,b \right\rangle =0$,
otherwise $c$ can be found from one of the pictures:

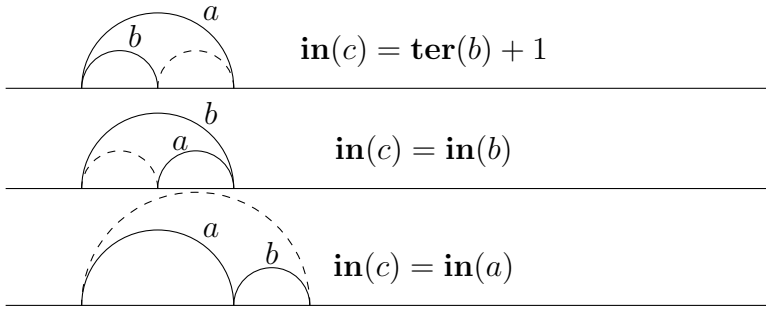
\begin{tikzpicture}
\draw (3,0) arc (0:180:1cm);
\draw (2,0) arc (0:180:0.5cm);
\draw [dashed] (3,0) arc (0:180:0.5cm);
\draw[-,>=stealth] (0,0)--(10,0);
\draw (2.7,1) node {$a$};
\draw (1.7,0.7) node {$b$};
\draw (5.5,0.5) node {$\In(c)=\Ter(b)+1$};
\end{tikzpicture}

\begin{tikzpicture}
\draw (3,0) arc (0:180:1cm);
\draw (3,0) arc (0:180:0.5cm);
\draw [dashed] (2,0) arc (0:180:0.5cm);
\draw[-,>=stealth] (0,0)--(10,0);
\draw (2.7,1) node {$b$};
\draw (2.3,0.6) node {$a$};
\draw (5.5,0.5) node {$\In(c)=\In(b)$};
\end{tikzpicture}

\begin{tikzpicture}
\draw (3,0) arc (0:180:1cm);
\draw (4,0) arc (0:180:0.5cm);
\draw [dashed] (4,0) arc (0:180:1.5cm);
\draw[-,>=stealth] (0,0)--(10,0);
\draw (2.7,1) node {$a$};
\draw (3.5,0.7) node {$b$};
\draw (5.5,0.5) node {$\In(c)=\In(a)$};
\end{tikzpicture}

$d=b-\left\langle a,b\right\rangle a$ may be found by similar rules.
\end{lemma}

%The first case is responsible for pairs of indecomposable representations with non-trivial epimorphisms, while the second and the third ones correspond to pairs with non-trivial monomorphisms and extensions respectively. -- Trash. EG

\begin{theorem}
Given a distinguished basis $A$, either $\alpha_i^{2}(A)=A$ or $\alpha_i^3(A)=A$.
\end{theorem}

\begin{proof}
Suppose that $A=(a_1,\ldots,a_n)$. Consider the sublattice generated by $a_i$ 
and $a_{i+1}$. If $\langle a_i,a_{i+1}\rangle=0$, then it has type $A_1\oplus A_1$, otherwise it has type $A_{2}$.
In the first case, $\alpha_i$ exchanges  $a_i$ and $a_{i+1}$ and has order 2. In the second case, it has order 3,
since $\alpha_1^3=\Id$ for $A_2$ (see Figure \ref{fig:bases for a2}).
\end{proof}

The braid group action on distinguished bases for $A_3$ is shown in Figure \ref{fig:bases for a3}.
The corresponding action on $\PF_3$ is shown in Figure \ref{fig:pf for a3}.

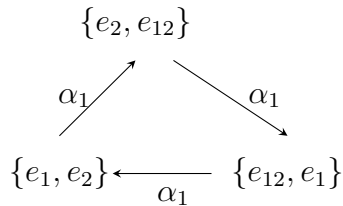
\begin{figure}[ht]
\begin{tikzpicture}
\draw (0,0) node {$\{e_1,e_2\}$};
\draw (1,2) node {$\{e_2,e_{12}\}$};
\draw (3,0) node {$\{e_{12},e_1\}$};
\draw [->,>=stealth] (0,0.5)--(1,1.5);
\draw [->,>=stealth] (1.5,1.5)--(3,0.5);
\draw [->,>=stealth] (2,0)--(0.7,0);
\draw (0.2,1) node {$\alpha_1$};
\draw (2.7,1) node {$\alpha_1$};
\draw (1.5,-0.3) node {$\alpha_1$};
\end{tikzpicture}
\caption{Braid group action for $A_2$}
\label{fig:bases for a2}
\end{figure}

\begin{figure}[ht]
\begin{tikzpicture}
\draw (0,0) node {$\{e_{13},e_2,e_{12}\}$};
\draw (3,0) node {$\{e_{13},e_{12},e_{1}\}$};
\draw (6,0) node {$\{e_{3},e_{13},e_{1}\}$};
\draw (9,0) node {$\{e_3,e_1,e_{23}\}$};
\draw (0,2) node {$\{e_{13},e_1,e_2\}$};
\draw (0,4) node {$\{e_{1},e_{23},e_{2}\}$};
\draw (0,6) node {$\{e_{1},e_3,e_{23}\}$};
\draw (3,6) node {$\{e_{1},e_2,e_{3}\}$};
\draw (6,6) node {$\{e_{2},e_{12},e_{3}\}$};
\draw (9,6) node {$\{e_{2},e_{13},e_{12}\}$};
\draw (9,4) node {$\{e_{2},e_3,e_{13}\}$};
\draw (9,2) node {$\{e_{3},e_{23},e_{13}\}$};
\draw (4.5,-1) node {$\{e_{12},e_3,e_{1}\}$};
\draw (4.5,7) node {$\{e_{12},e_1,e_{3}\}$};
\draw (2,3) node {$\{e_{23},e_{13},e_{2}\}$};
\draw (7,3) node {$\{e_{23},e_2,e_{13}\}$};

\draw [->,>=stealth,dashed] (1.2,0)--(1.8,0);
\draw [<-,>=stealth] (4.2,0)--(4.8,0);
\draw [->,>=stealth,dashed] (7.2,0)--(7.8,0);
\draw [<-,>=stealth,dashed] (1.2,6)--(1.8,6);
\draw [->,>=stealth] (4.2,6)--(4.8,6);
\draw [<-,>=stealth,dashed] (7.2,6)--(7.8,6);
\draw [<-,>=stealth,dashed] (0,0.2)--(0,1.8);
\draw [->,>=stealth] (0,2.2)--(0,3.8);
\draw [<-,>=stealth,dashed] (0,4.2)--(0,5.8);
\draw [->,>=stealth,dashed] (9,0.2)--(9,1.8);
\draw [<-,>=stealth] (9,2.2)--(9,3.8);
\draw [->,>=stealth,dashed] (9,4.2)--(9,5.8);
\draw [<-,>=stealth,dashed] (0.5,1.8)--(3,0.2);
\draw [->,>=stealth,dashed] (0.5,4.2)--(3,5.8);
\draw [->,>=stealth,dashed] (8.5,1.8)--(6,0.2);
\draw [<-,>=stealth,dashed] (8.5,4.2)--(6,5.8);
\draw [<->,>=stealth] (0.5,0.2)--(8.5,5.8);
\draw [<->,>=stealth] (0.5,5.8)--(8.5,0.2);
\draw [<->,>=stealth,dashed] (4.5,-0.8)--(4.5,6.8);
\draw [<->,>=stealth,dashed] (3,3)--(6,3);
\draw [->,>=stealth] (3,-0.2)--(4,-0.8);
\draw [->,>=stealth] (5,-0.8)--(6,-0.2);
\draw [<-,>=stealth] (3,6.2)--(4,6.8);
\draw [<-,>=stealth] (5,6.8)--(6,6.2);
\draw [<-,>=stealth] (0.5,2.2)--(1.5,2.8);
\draw [<-,>=stealth] (1.5,3.2)--(0.5,3.8);
\draw [->,>=stealth] (8.5,2.2)--(7.5,2.8);
\draw [->,>=stealth] (7.5,3.2)--(8.5,3.8);
\end{tikzpicture}
\caption{Braid group action for $A_3$: $\alpha_1$ is solid, $\alpha_2$ is dashed}
\label{fig:bases for a3}
\end{figure}
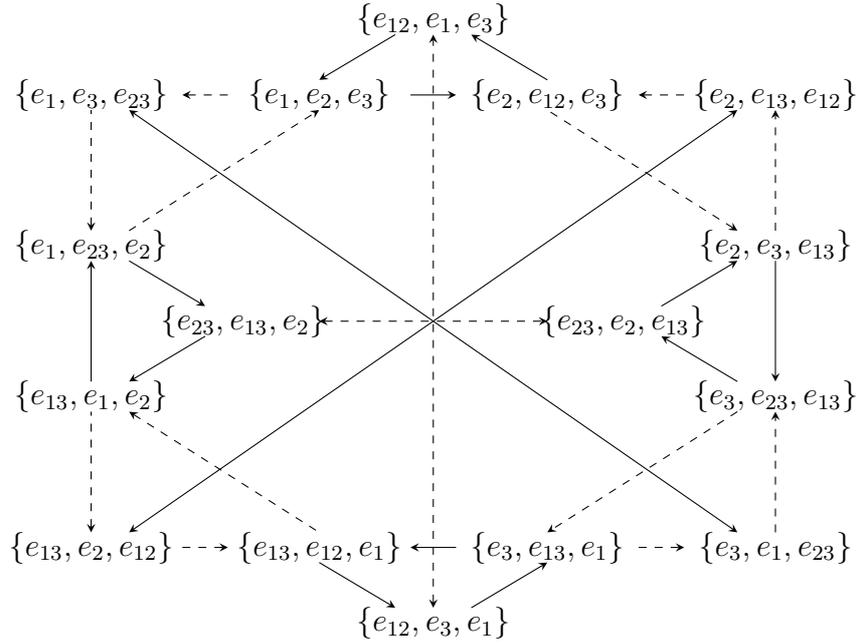

\begin{theorem}
\label{mut}
Consider a parking function diagram $D.$ The action of the braid group is determined by the action of generators $\alpha_k$ and $\beta_k.$ There are the following cases:
\begin{itemize}
\item[1)]
The segment $P_{k+1} P_k$ goes strictly SW--NE and all its common points with $D(E)$ are $P_l,$ where $l \leq (k+1).$ Then 
\begin{itemize}
\item{} $\alpha_k$ exchanges the numbers $k$ and $(k+1)$ ;
\item{} $\beta_k$ replaces the row of $D$ containing $P_k$ with a row of length being equal to the $x$-coordinate of $P_{k+1},$ inserting it directly below the row containing $P_{k+1}.$ All rows between the old and the new ones are shifted together with their numbers by one position above.
\end{itemize}
\item[2)]
The segment $P_k P_{k+1}$ goes strictly SW--NE and all its common points with $D(E)$ are $P_l,$ where $l \leq (k+1).$ Then
\begin{itemize}
\item{}$\alpha_k$ replaces the row of $D$ containing $P_{k+1}$ with a row of length being equal to the $x$-coordinate of $P_{k},$ inserting it directly below the row containing $P_{k}.$  All rows between the old and the new ones are shifted together with their numbers by one position above.
\item{}$\beta_k$ exchanges the numbers $k$ and $(k+1)$.
\end{itemize}
\item[3)]
$P_k$ and $P_{k+1}$ have the same $x$-coordinate. Start from $P_{k+1}$ and go strictly north-east until we meet first $P_l$ s.t. $l > k$ or encounter the Dyck path or the $x$-axis. Denote the point where we stop by $Q.$ Then
\begin{itemize}
\item{}$\alpha_k$ replaces the row of $D$ containing $P_{k}$ with a row of length being equal to the $x$-coordinate of $Q$ inserting it directly below the row containing $Q$ (or the x-axis, if that contains $Q$). All rows between the old and the new ones are shifted together with their numbers by one position below.
\item{}$\beta_k$ replaces a row of $D$ containing $P_{k+1}$ with a row of length being equal to the $x$-coordinate of $Q$ putting it directly below the row containing $Q$ (or the x-axis, if that contains $Q$). All rows between the old and the new ones are shifted together with their numbers by one position below.
\end{itemize}
\item[4)]
In the other cases both $\alpha_k$ and $\beta_k$ exchange the numbers $k$ and $(k+1)$.
\end{itemize}
\end{theorem}

Theorem \ref{mut} immediately follows from Lemma \ref{arcsmut} and Theorem \ref{T2}.
An example of the action of $\alpha_i$ and $\beta_i$ on parking function diagrams is shown in
Figure \ref{fig:alpha 6}.

%Note that the condition concerning $P_l$ in cases 1) and 2) is important. We illustrate this with the following figure.

%\begin{tikzpicture}[line width=0.6pt, scale=0.6]

%  \draw [<->,thick] (1,4) node (yaxis) [above] {$y$}
%     |- (5,3) node (xaxis) [right] {$x$};
%  \draw [thick] (1,-1) -- (1,3);
%  \draw [thick] (0,3) -- (1,3);
%  \draw (3,3) -- node[right]{2}(3,2) -- (2,2) -- node[right]{3}(2,1) -- (1,1) -- node[right]{1}(1,0);
%  \draw (1,2) -- (2,2) -- (2,3);
%  \draw (4,3) arc (0:180:0.5cm);
%  \draw (4,3) arc (0:180:1cm);
%  \draw (2,3) arc (0:180:0.5cm);
%  \draw (1.5,3.9) node {1};
%  \draw (3,3.6) node {2};
%  \draw (2.8,4.3) node {3};
  
%\draw[->,>=stealth,thick] (6,2)--node[above]{$\alpha_1 = \beta_1$}(9,2); 
%\draw[->,>=stealth,thick] (9,1)--node[below]{$\alpha_1 = \beta_1$}(6,1);
    
%   \draw [<->,thick] (11,4) node (yaxis) [above] {$y$}
%     |- (15,3) node (xaxis) [right] {$x$};
%  \draw [thick] (11,-1) -- (11,3);
%  \draw [thick] (10,3) -- (11,3);
%  \draw (13,3) -- node[right]{1}(13,2) -- (12,2) -- node[right]{3}(12,1) -- (11,1) -- node[right]{2}(11,0);
%  \draw (11,2) -- (12,2) -- (12,3);
%  \draw (14,3) arc (0:180:0.5cm);
%  \draw (14,3) arc (0:180:1cm);
%  \draw (12,3) arc (0:180:0.5cm);
%  \draw (11.5,3.9) node {2};
%  \draw (13,3.6) node {1};
%  \draw (12.8,4.3) node {3};

%\end{tikzpicture}

\begin{figure}[ht]
\begin{tikzpicture}[scale=0.6, line width=0.6pt]
  \draw [<->,thick] (1,16) node (yaxis) [above] {$y$}
     |- (10,14) node (xaxis) [right] {$x$};
  \draw [thick] (1,5) -- (1,14);
  \draw [thick] (0,14) -- (1,14);
  \draw (1,13) -- (8,13) --node [right]{\bf 2} (8,14);
  \draw (1,12) -- (7,12) --node [right]{\bf 7}(7,13) -- (7,14);
  \draw (1,11) -- (6,11) --node[right]{\bf 3}(6,12) -- (6,14);
  \draw (1,10) -- (5,10) --node[right]{\bf 5}(5,11) -- (5,14);
  \draw (1,9) -- (4,9) --node[right]{\bf 4}(4,10) -- (4,14);
  \draw (1,8) -- (2,8) --node[right]{\bf 1}(2,9) -- (2,14);
  \draw (1,7) --node[right]{\bf 8}(1,8);
  \draw (1,6) --node[right]{\bf 6}(1,7);
  \draw (3,9) -- (3,14);

\draw (9,14) arc (0:180:1cm);
\draw (9,14) arc (0:180:0.5cm);
\draw (7,14) arc (0:180:0.5cm);
\draw (7,14) arc (0:180:1cm);
\draw (7,14) arc (0:180:3cm);
\draw (5,14) arc (0:180:0.5cm);
\draw (3,14) arc (0:180:0.5cm);
\draw (3,14) arc (0:180:1cm);
%\draw[-,>=stealth] (-3,0)--(11,0);
\draw (1.9,14.5) node {1};
\draw (7.9,14.5) node {2};
\draw (5.9,14.5) node {3};
\draw (4.5,14.9) node {4};
\draw (6,15.5) node {5};
\draw (4,17.5) node {6};
\draw (8,15.5) node {7};
\draw (2.5,15.4) node {8};

  \draw [<->,thick] (1,-2) node (yaxis) [above] {$y$}
     |- (10,-4) node (xaxis) [right] {$x$};
  \draw [thick] (1,-13) -- (1,-4);
  \draw [thick] (0,-4) -- (1,-4);
  \draw (1,-5) -- (8,-5) --node [right]{\bf 2} (8,-4);
  \draw (1,-6) -- (7,-6) --node [right]{\bf 6}(7,-5) -- (7,-4);
  \draw (1,-7) -- (6,-7) --node[right]{\bf 3}(6,-6) -- (6,-4);
  \draw (1,-8) -- (5,-8) --node[right]{\bf 5}(5,-7) -- (5,-4);
  \draw (1,-9) -- (4,-9) --node[right]{\bf 4}(4,-8) -- (4,-4);
  \draw (1,-10) -- (2,-10) --node[right]{\bf 1}(2,-9) -- (2,-4);
  \draw (1,-11) --node[right]{\bf 8}(1,-10);
  \draw (1,-12) --node[right]{\bf 7}(1,-11);
  \draw (3,-9) -- (3,-4);

\draw (9,-4) arc (0:180:0.5cm);
\draw (9,-4) arc (0:180:1cm);
\draw (9,-4) arc (0:180:4cm);
\draw (7,-4) arc (0:180:0.5cm);
\draw (7,-4) arc (0:180:1cm);
\draw (5,-4) arc (0:180:0.5cm);
\draw (3,-4) arc (0:180:0.5cm);
\draw (3,-4) arc (0:180:1cm);
%\draw[-,>=stealth] (-3,0)--(11,0);
\draw (1.9,-3.5) node {1};
\draw (7.9,-3.5) node {2};
\draw (5.9,-3.5) node {3};
\draw (4.5,-3.1) node {4};
\draw (6,-2.5) node {5};
\draw (8,-2.5) node {6};
\draw (4.5,0.5) node {7};
\draw (2.5,-2.6) node {8};

    \draw [<->,thick] (16,7) node (yaxis) [above] {$y$}
     |- (25,5) node (xaxis) [right] {$x$};
  \draw [thick] (16,-4) -- (16,5);
  \draw [thick] (15,5) -- (16,5);
  \draw (16,4) -- (23,4) --node [right]{\bf 2} (23,5);
  %\draw (16,3) -- (22,3) --node [right]{\bf 3}(22,4) -- (22,5);
  \draw (16,3) -- (21,3) --node[right]{\bf 3}(21,4) -- (21,5);
  \draw (16,2) -- (20,2) --node[right]{\bf 5}(20,3) -- (20,5);
  \draw (16,1) -- (19,1) --node[right]{\bf 4}(19,2) -- (19,5);
  \draw (16,0) -- (17,0) --node[right]{\bf 1}(17,1) -- (17,5);
  \draw (16,-1) --node[right]{\bf 8}(16,0);
  \draw (16,-2) --node[right]{\bf 7}(16,-1);
  \draw (16,-3) --node[right]{\bf 6}(16,-2);
  \draw (18,1) -- (18,5);
  \draw (22,4) -- (22,5);

\draw (24,5) arc (0:180:0.5cm);
%\draw (24,5) arc (0:180:1cm);
\draw (24,5) arc (0:180:4cm);
\draw (22,5) arc (0:180:0.5cm);
\draw (22,5) arc (0:180:1cm);
\draw (22,5) arc (0:180:3cm);
\draw (20,5) arc (0:180:0.5cm);
\draw (18,5) arc (0:180:0.5cm);
\draw (18,5) arc (0:180:1cm);
%\draw[-,>=stealth] (-3,0)--(11,0);
\draw (16.9,5.5) node {1};
\draw (22.9,5.5) node {2};
\draw (20.9,5.5) node {3};
\draw (19.5,5.9) node {4};
\draw (21,6.5) node {5};
\draw (19,9.5) node {6};
\draw (21.2,7.5) node {7};
\draw (17.5,6.4) node {8};

\draw[->,>=stealth,thick] (4,5)--node[left]{$\beta_6$}(4,2);
\draw[->,>=stealth,thick] (5,2)--node[right]{$\alpha_6$}(5,5);
\draw[->,>=stealth,thick] (13,1)--node[above]{$\alpha_6$}(10,-1);
\draw[->,>=stealth,thick] (10,-2)--node[below]{$\beta_6$}(13,0);
\draw[->,>=stealth,thick] (13,6)--node[above]{$\beta_6$}(10,8);
\draw[->,>=stealth,thick] (10,7)--node[below]{$\alpha_6$}(13,5);

\end{tikzpicture}
\caption{Action of $\alpha_6$ and $\beta_6$ on parking function diagrams}
\label{fig:alpha 6}
\end{figure}
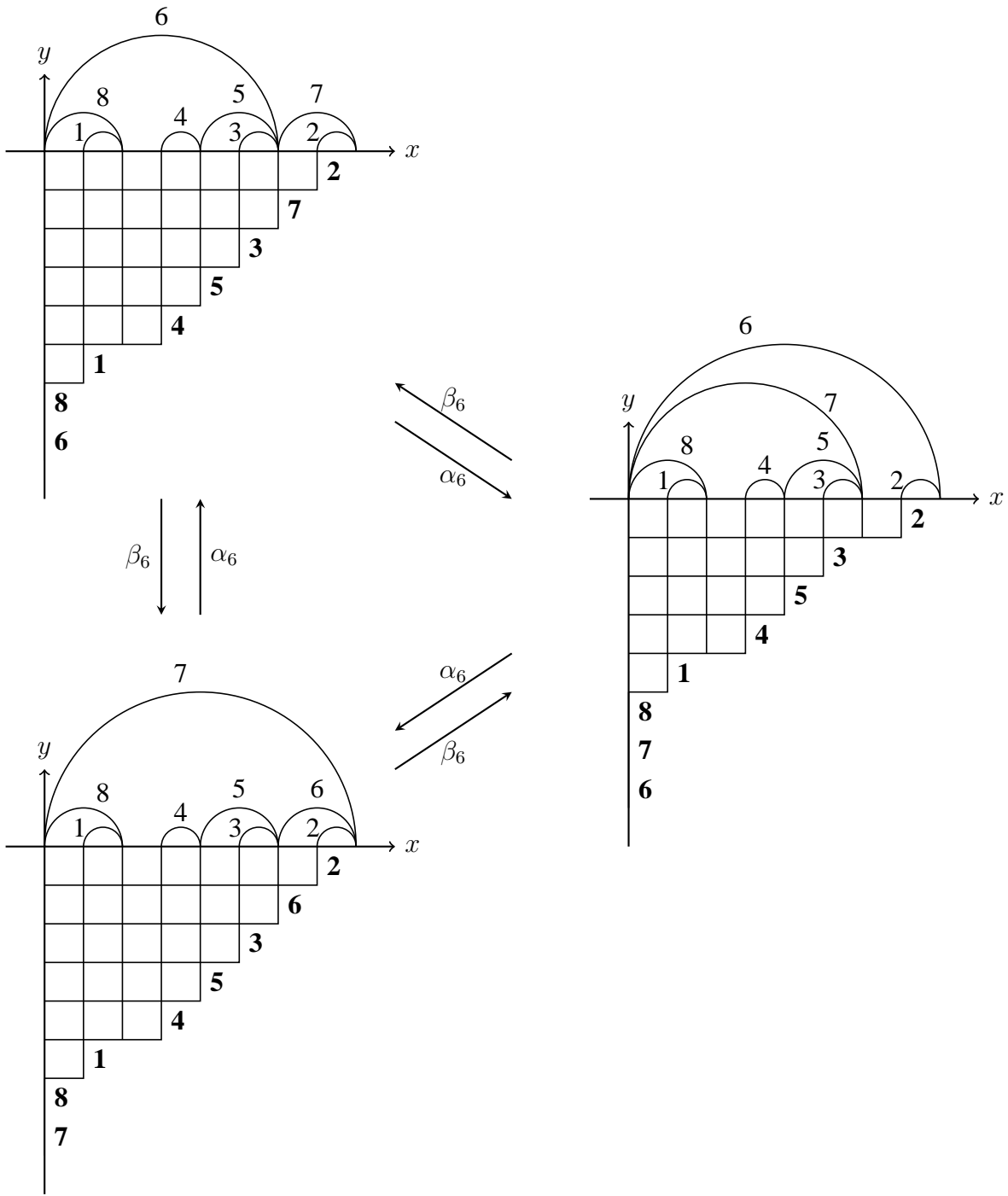

Now we describe a relation of the action of $\alpha_k$ and $\beta_k$ on a diagram of a non-decreasing parking function to the following operations on Young diagrams which were introduced in \cite{Gor}.

\begin{definition} 
Let $M = (\mu_1,\mu_2, \ldots)$ and  $N$ be Young diagrams. We will say that $N$
is obtained from $M$ by a \textit{flip in row $k$}, if we can obtain it from $M$ by throwing out row number $k$  (it can have length $0$) and insertion of another row of length $l$ %in such place that it will be Young diagram; 
where $l$ is defined by the following rule:

Start from the point $(\mu_k, -k)$ of diagram $M$ and go along the line $x - y = \mu_k + k.$ If $k$-th row is longer than $(k+1)$-st, we should go in SW direction, if their lengths are the same -- in NE direction. Stop at the first moment, when we touch the boundary of $M$ or the coordinate line. The $x$-coordinate of this point will be $l.$

\end{definition}

These flips are induced by flips of triangulations of the $(n+2)-$gon  via bijection to $\mathbb{Y}_n.$ They also correspond (\cite{Gor}) to mutations in cluster algebras of type $A_{n-1}.$ %For the details see \cite{Gor}.

Denote by $Y(D)$ the Young diagram corresponding to a parking function diagram $D.$

\begin{theorem}
Let $D$ be the diagram of a non-decreasing parking function on $n$ elements. Then $Y(D) \in \mathbb{Y}_n$ and  $Y(\alpha_k)$ either equals $Y(D)$ or it is obtained from $Y(D)$ by a single flip in some row. The same holds for $Y(\beta_k).$
\end{theorem}

\begin{proof}
Remark that for non-decreasing parking function for every number $a$ the $y$-coordinate of $P_a$ equals to $(a - n).$
It remains to apply Theorem \ref{mut}.
\end{proof}

\section{Quiver representations}
\label{sec:quiver}

%Here I insolently use Keller's survey. Should change a little bit. - M.G.

Let us briefly repeat the main notions of quiver representation theory, following
\cite{keller} and \cite{kirillov}.

\begin{definition}
A {\it quiver} is an oriented graph. It is {\it finite}, if it has finite number of vertices and finite number of arrows. Let $Q$ be a finite quiver without oriented cycles and let $k$ be an algebraically closed field. A {\it representation} of $Q$ is the datum $V$ of a finite-dimensional vector space $V_i$ over $k$ for each vertex $i$ of $Q$
and a linear map $V_\alpha : V_i \rightarrow V_j$ for each arrow $\alpha : i \rightarrow j$ of $Q.$
 
The \textit{dimension vector} of a representation $V$ is the sequence $\underline{\dim} V$ of dimensions $\dim V_i, i \in Q_0.$ A representation $V$ is \textit{indecomposable} if it is non zero and in each decomposition $V = V' \oplus V''$ we have $V' = 0$ or $V'' = 0.$
A quiver is called \textit{representation-finite} if it has only finitely many isomorphism classes of indecomposable representations.
\end{definition}

It is a well-known fact that representations of $Q$ over $k$ form an abelian category which is denoted by $\Rep_k Q.$ Moreover, this category is {\it hereditary}, that is
$$\Ext^i(V,W) = 0, \quad \forall i>1, V, W \in \Rep_k Q.$$

\begin{theorem} \label{Gab} {\normalfont{(Gabriel \cite{Gab})}}. Let $Q$ be a connected quiver and assume that $k$ is algebraically closed. $Q$ is representation-finite if and only if the underlying graph of $Q$ is a simply laced Dynkin diagram $\Delta.$ In this case, the map taking a representation with dimension vector $(d_i)$ to the root $\sum d_i \alpha_i$ of the root system associated with $\Delta$ yields a bijection from the set of isomorphism classes of indecomposable representations to the set of positive roots.
\end{theorem}

\begin{definition}
The {\it Euler form} is the non-symmetric form on $K_0(\Rep_k  Q)$ defined as
$$\left\langle E,F\right\rangle=\dim \Hom(E,F)-\dim \Ext^{1}(E,F).$$
Sometimes it is also called the {\it Ringel form}.
\end{definition}

\begin{proposition}
Suppose that $Q$ satisfies the conditions of Theorem \ref{Gab}, and $E_{\alpha}$, $E_{\beta}$ are two indecomposable representations corresponding to roots $\alpha$ and $\beta$. Then
$$\left\langle E_{\alpha},E_{\beta}\right\rangle+\left\langle E_{\beta},E_{\alpha}\right\rangle=(\alpha,\beta).$$
\end{proposition}

From now on we will study the representations of the $A_n$ quiver with the following orientation:
%\begin{center}
%\begin{tikzpicture}

% \draw[->,>=stealth] (0,0) node [below left] {1} node [left] {$\bullet$}  -- (0.95,0);
%  \draw[->,>=stealth] (1,0) node [below left] {2} node [left] {$\bullet$}  -- (2,0);
%  \draw (2.2,0) node {$. . .$};
%  \draw (2.4,0) node {$\bullet$};
%  \draw (2.6,0) node {$\bullet$};
%  \draw[->,>=stealth] (2.8,0) -- (3.75,0) node [below] {\large{n}};
%  \fill[black] (1,0) circle (2pt);
%  \fill[black] (0,0) circle (2pt);
%  \fill[black] (3.8,0) circle (2pt);
%\end{tikzpicture}
%\end{center}

\begin{center}
\begin{tikzpicture}[scale=0.5]
 \draw[->,>=stealth] (1,1) node [left] {$\bullet$} node [below left] {1} -- (2,1) node [right] {$\bullet$} node [below right] {2};
% \draw[->,>=stealth] (3,1)-- (4,1) node [right] {$\bullet$};
% \draw[->,>=stealth] (5,1)-- (6,1) node [right] {$\bullet$};
 \draw[->,>=stealth] (3,1)-- (4,1) node [right] {$. . .$};
% \draw[->,>=stealth] (9.5,1)-- (10.5,1) node [right] {$\bullet$};
 \draw[->,>=stealth] (5.5,1)-- (6.5,1) node [right] {$\bullet$} node [below right] {$n$};
\end{tikzpicture}
\end{center}

\begin{definition}
An indecomposable representation $E$ is called {\it exceptional}, if
$\Ext^1(E,E) = 0,$ and a sequence $(E_1,\ldots,E_r)$ of exceptional representations is said to be
an {\it exceptional sequence}, if $$\Hom(E_j ,E_i) = 0 = \Ext^1(E_j ,E_i)\quad \mbox{\rm for}\quad j > i.$$
%In some literature there is an alternative name {\it exceptional collections}.
An exceptional sequence is {\it complete}, if $r = n.$
\end{definition}

\begin{lemma} \label{esdb}
1. Every indecomposable representation of the $A_n$ quiver is exceptional.

2. The complete exceptional sequences for $A_n$ quiver are in 1-to-1 correspondence with the distinguished bases of positive roots under Gabriel's bijection.
% Bijection sends sequence of indecomposable representations to the sequence of positive roots with the same supports.
\end{lemma}

\begin{proof}
Let $E$ be an indecomposable representation. Then, by Gabriel's theorem, $\left\langle E,E \right\rangle =1$ and $\dim \Hom(E,E)=1.$
Therefore, $\dim \Ext^{1}(E,E)=0$, so $E$ is exceptional.

One can check that the Euler form corresponds to the Seifert form on the root system under the Gabriel's bijection; therefore, every exceptional sequence corresponds to a distinguished basis. To prove the converse, one has to check that for indecomposable representations $E, F$ we have
$$\left\langle E,F\right\rangle=0\quad \Rightarrow\quad \Hom(E,F)=\Ext^{1}(E,F)=0.$$
This follows from Lemmas \ref{L1} and \ref{dbarcs}.
\end{proof}

Moreover, we can understand dimensions of extensions and morphism  spaces between objects in complete exceptional sequences in terms of the corresponding distinguished bases and, therefore, in terms of parking functions.

\begin{lemma}(compare with Lemma \ref{dbarcs}).
Consider two indecomposable representations $V, W$ such that $\left\langle W,V\right\rangle =0$ and the two corresponding vectors $a = e_{\In(a)}+\ldots+e_{\Ter(a)}, b = e_{\In(b)}+\ldots+e_{\Ter(b)}.$ Then the following statements hold:

\begin{enumerate}
\item[1)] 
%$\mbox{dim}(\mbox{Epi}(V,W)) = \delta_{\In(a)}^{\In(b)};$
%\item[2)]
%$\mbox{dim}(\mbox{Mono}(V,W)) = \delta_{j(a)}^{j(b)};$
If $\In(a)=\In(b)$ then $W$ is a quotient of $V$; if $\Ter(a)=\Ter(b)$ then $V$ is a subrepresentation of $W$.
In both cases $\dim \Hom(V,W)=1$; in the first case every nontrivial morphism from $V$ to $W$ is surjective, in the second case -- injective. In the other cases $\dim \Hom(V,W)=0.$
\item[2)]
$\dim \Ext^{1}(V,W) = \delta_{\Ter(a)+1}^{\In(b)}.$
\end{enumerate}
\end{lemma}

\begin{corollary}
Consider an exceptional sequence $E = (E_1,\ldots,E_n)$ and two of its elements $E_i, E_j, i < j.$ Let $D(E)$ be corresponding parking function diagram obtained by combining Gabriel's bijection and the ``initial vector'' map. Let $P_k$ be the points on $D(E)$ (see Theorem \ref{T2}). Then the following statements hold:

\begin{enumerate}
\item[1)]
$E_j$ is a subrepresentation of $E_i$ iff $P_j$ and $P_i$ have the same $x$-coordinate. In this case, $\dim \Hom(E_i,E_j)=1$ and every nontrivial morphism from $E_i$ to $E_j$ is surjective.
\item[2)]
$E_i$ is a subrepresentation of $E_j$ iff the segment $P_j P_i$ goes strictly SW--NE and all its common points with $D(E)$ are $P_l,$ where either $l \leq i$ or $l = j.$
In this case, $\dim \Hom(E_i,E_j)=1$ and every nontrivial morphism from $E_i$ to $E_j$ is injective.
\item[3)]
$\dim \Ext^{1}(E_i,E_j) = 1$ iff there exists $P_k, k > i$ s.t. a segment $P_i P_l$ goes strictly SW--NE and all its common points with $D(E)$ are $P_l,$ where either $l \leq i$ or $l = k;$ and 
$P_k$ and $P_j$ have the same $x$-coordinate. Otherwise, $\dim \Ext^{1}(E_i,E_j) = 0.$
\item[4)]
$\dim \Hom(E_i,E_j) = 0$ except in cases 1) and 2).
\end{enumerate}
\end{corollary}

The following theorem is a reformulation of the Theorem \ref{Tmain}.

\begin{theorem}
There exists a filtration $F$ on $K_0(\Rep A_n)$ such that every complete exceptional sequence can be uniquely
reconstructed from the sequence of filtration levels of its representations, and such sequences are in one-to-one correspondence with the
parking functions.
\end{theorem}

\begin{proof}
Consider a filtration $F=\{F_k\}$ such that $F_{k}$ is spanned by the simple representations $S_k,\ldots,S_n$ corresponding to the simple roots $e_k,\ldots, e_n$.
It remains to apply Theorem \ref{Tmain}.
\end{proof}

\begin{definition}
Let us call a distinguished basis $A$  {\it non-decreasing}, if the parking function $\In(A)$ is non-decreasing.
\end{definition}

%\begin{lemma} \label{ndb}
%Distinguished basis $A = (a_1,\ldots,a_n)$ is non-decreasing if and only if following conditions hold:

%\begin{itemize}
%\item[1)] There are no such $i \neq j, \quad i,j \in [1,n],$ that 
%$$a_i = e_k + \ldots + e_l; a_j = e_m + \ldots + e_l.$$

%Equivalently, if 
%$$\left\langle a_i, a_j \right\rangle = 1,$$

%then support of $a_i$ contains support of $a_j.$

%\item[2)] If 
%$$\left\langle a_i, a_j \right\rangle = 0, \quad a_i = e_k + \ldots + e_l, \quad a_j = e_m + \ldots + e_q, \quad m > k;$$
%then $j > i.$
%\end{itemize}
%\end{lemma}

%\begin{proof}
%By direct verification of cases from Lemma \ref{L1}.
%\end{proof}

%The next corollary is just reformulation of this lemma in terms of arcs.

\begin{lemma} \label{nddbarcs}
Non-decreasing distinguished bases correspond bijectively to ordered collections of $n$ pairwise non-intersecting arcs with properties 1), 3) and 4) from Proposition \ref{dbarcs} s.t. there are no pairs of arcs with the same right ends. Equivalently, non-decreasing distinguished bases correspond bijectively to unordered sets of non-intersecting arcs with the property 4) s.t. there are no pairs of arcs with the same %(common?) 
right ends. 
\end{lemma}

\begin{definition}
{\it Non-decreasing exceptional collections} of type $A_n$ are exceptional collections of representations $X = \left\{X_1, \ldots, X_n \right\}$ of the $A_n$ quiver satisfying one of the following equivalent properties:

\begin{itemize}
%\item[(i)] For each $i \neq j$  either $\mbox{Ext}^1 (X_i, X_j) = 0 = \mbox{Hom}(X_i, x_j),$ or $\mbox{Ext}^1 (X_j, X_i) = 0 = \mbox{Hom}%(X_i, X_j).$

\item[(i)] There are no monomorphisms $X_i \hookrightarrow X_j, i \neq j.$

%\item[(ii)] There are no monomorphisms $X \hookrightarrow X_j$, where $X$ is a subrepresentation of $X_i, i \neq j.$
%Looks suspicious - suppose that X_i is very large.....-EG

%\item[(iii)] $X_1, \ldots, X_n$ are linearly independent and generate the whole category.

%Possibly the condition (i) should be replaced by one of the next:

\item[(ii)] If $\left\langle X_i, X_j \right\rangle \neq 0,$ then $X_i$ is not a subrepresentation of $X_j.$
\end{itemize}
\end{definition}

%\begin{proposition}
%There is a natural bijection from the set of non-decreasing distinguished bases of the system $A_n$ to the set of non-decreasing exceptional collections of type $A_n.$ 
%\end{proposition}

%\begin{proof}
%By Proposition..., distinguished bases correspond naturally to exceptional sequences, that is union of properties (i) and (iii). Condition 1) in Lemma \ref{ndb} provides exactly property (ii) (or (ii'), or (ii''), which are equivalent to (ii) in case of type $A_n$) for these collections. Exceptional sequences are ordered, and one set of objects may be ordered in different ways to obtain different exceptional secuences of its elements. Condition 2) in Lemma \ref{ndb} provides unique order on these sets, so exactly one non-decreasing distinguished basis corresponds to each such set. With property (ii), these sets are exactly non-decreasing exceptional collections.
%\end{proof} %need cleanup

%The next lemma is obvious.

\begin{lemma}
Non-decreasing exceptional collections correspond bijectively to sets of arcs from Lemma \ref{nddbarcs}.
\end{lemma} 

\begin{corollary}
The following objects are in 1-to-1 correspondence to each other:
\begin{itemize}
\item[(i)] Non-decreasing distinguished bases of the root system $A_n;$

\item[(ii)] Young diagrams from $\mathbb{Y}_n;$

\item[(iii)] Dyck paths of the length $2n;$

\item[(iv)] Non-decreasing exceptional collections of type $A_n;$

\item[(v)] Sets of arcs described in Lemma \ref{nddbarcs}.

%\item[(vi)] Clusters of an arbitrary cluster algebra of type $A_{n-1}.$  %Definition????? EG
\end{itemize}
The cardinality of all these sets is equal to $n-$th Catalan number $c_n.$

\end{corollary}

The set of exceptional sequences of representations of the $A_n$ quiver carries an action of the braid group (see \cite{GoRu,buan2,CB,R}) which corresponds to the above action under the  bijection from Lemma \ref{esdb}.

\section{Non-crossing partitions}

In this section, we reinterpret the results of \cite{stanley} in terms of the ``initial vector'' map.

\begin{definition} 
A {\it non-crossing partition} of the set $\{0,1,\ldots,n\}$ is a partition $\pi$ such that if $a<b<c<d$ and some block $B$ of $\pi$ 
contains both $a$ and $c$, while some block $B'$ of $\pi$ contains both $b$ and $d$, then $B=B'$.
\end{definition}

Non-crossing partitions form a partially ordered set: we say that $\pi\preceq \sigma$, if $\pi$ is a refinement of $\sigma$.

\begin{lemma}
Let $A=(a_1,\ldots,a_n)$ be a distinguished basis. Let us draw the set of arcs (see Proposition \ref{dbarcs}) corresponding to $a_1,\ldots,a_k$,
and define a partition $\pi_k$ of the set $\{0,1,\ldots,n\}$ into connected components. Then $\Pi(A)=\{\pi_k\}$ is a maximal chain of 
non-crossing partitions.
\end{lemma}

\begin{proof}
It follows from Lemma \ref{dbarcs} that $\pi_k$ is a non-crossing partition for all $k$: arcs do not intersect, hence if
$a<b<c<d$ and $a,c$ are connected while $b,d$ are connected too, then, by Jordan's theorem, $a,b,c$ and $d$ belong to the same connected component.

If two points are connected in $\pi_k$, then they are connected in $\pi_m$ for $m\ge k$; therefore, $\pi_k$ is a refinement of $\pi_m$.
In rests to note that every maximal chain contains $n+1$ non-crossing partitions and its $\preceq$-minimal element is a partition into singletons. 
\end{proof}

\begin{remark}
This lemma seems to be parallel to the constructions of \cite{buan2}, where connected components of the ``Hom--Ext'' quivers for 
exceptional collections were considered.
\end{remark}

Suppose that a partition $\sigma$ is an immediate successor of a partition $\pi$,
in other words, it is obtained from $\pi$ by merging two blocks $B$ and $B'$. Suppose that $\min (B)<\min (B')$, and define
$$\Lambda(\pi,\sigma)=\max\{i\in B|i<B'\}.$$

\begin{definition} 
Let $\Pi=\pi_0\prec \pi_1\prec\ldots \prec \pi_n$ be a maximal chain of non-crossing partitions.
Define
$$\Lambda(\pi)=(\Lambda(\pi_0,\pi_1),\ldots,\Lambda(\pi_{n-1},\pi_n)).$$
\end{definition}

\begin{theorem} (\cite{stanley}) The map $\Lambda$ is a bijection between the set of maximal chains of non-crossing partitions of the set $\left\{0,\ldots,n\right\}$and the set of parking function on $n$ elements.
\end{theorem}

\begin{theorem}
The map $\Pi$ is a bijection between the set of distinguished bases and maximal chains of non-crossing partitions,
and $$\Lambda(\Pi(A))+(1,\ldots,1)=\In(A)$$
for a distinguished basis $A$.
\end{theorem}

\begin{proof}
Given a maximal chain of non-crossing partitions 
$$\pi_0\prec \pi_1\prec\ldots \prec \pi_n,$$ 
let us construct a distinguished basis by induction.
Suppose that we already constructed roots $a_1,\ldots,a_k$ such that $\pi_i$ is a partition into connected components of $\{a_1,\ldots,a_i\}$ for
$i\le k$. Suppose that the partition $\pi_{k+1}$ is obtained from $\pi_k$ by a merge of two blocks $B$ and $B'$ and
$\min(B)<\min(B')$, let us construct a root $a_{k+1}$ joining $B$ and $B'$ is a single connected component.

Consider the root
$$a_{k+1}=e_{\Lambda(\pi_k,\pi_{k+1})+1}+\ldots+e_{\max(B')}.$$
One can check that all conditions of Lemma \ref{dbarcs} are satisfied, and 
$$\left\langle a_{k+1},a_i\right\rangle=0, \quad \forall i\leq k.$$
Therefore, $\Pi$ is surjective and 
$$\In(a_{k+1})=\Lambda(\pi_k,\pi_{k+1})+1.$$
\end{proof}

\begin{remark}
In \cite{lya} and \cite{loo} Lyashko and Looijenga obtained the following general formula for the number of distinguished bases 
for a singularity $A_{\mu}$, $D_{\mu}$ and $E_{\mu}$:
\begin{equation}
\label{nbases}
N_{bases}=\frac{\mu!h^{\mu}}{|W|},
\end{equation}
where $h$ is the Coxeter number and $|W|$ is the order of the corresponding Weyl group.
The notion of the non-crossing partition was generalized to a general Coxeter group by Reiner and Athanasiadis (\cite{anata,reiner}),
and Reading and Chapoton (\cite{readchains,chapoton}) proved that the number of maximal chains of non-crossing partitions is given by the formula (\ref{nbases}).
We plan to study the possible generalizations of parking functions for general Dynkin quivers in the future research.
\end{remark}

\medskip

\end{document}